\documentclass[10pt]{amsart}

\usepackage{geometry, amssymb}
\usepackage{graphicx}
\usepackage{color}

\newtheorem{teo}{Theorem}[section]
\newtheorem*{main}{Main Theorem}

\newtheorem{prop}[teo]{Proposition}
\newtheorem{lm}[teo]{Lemma}
\theoremstyle{definition}
\newtheorem{oss}[teo]{Remark}

\newtheorem*{ack}{Acknowledgments}

\usepackage[colorlinks=true,urlcolor=blue, citecolor=red,linkcolor=blue,
linktocpage,pdfpagelabels, bookmarksnumbered,bookmarksopen]{hyperref}
\usepackage[hyperpageref]{backref}

\numberwithin{equation}{section}

\begin{document}
\title[Principal frequencies, volume, inradius]{On principal frequencies, volume\\ and inradius in convex sets}

\subjclass{35P15, 49J40, 47A75}
\keywords{Convex sets, nonlinear eigenvalue problems, torsional rigidity, inradius.}

\author[Brasco]{Lorenzo Brasco}
\address[L.\ Brasco]{Dipartimento di Matematica e Informatica
\newline\indent
Universit\`a degli Studi di Ferrara
\newline\indent
Via Machiavelli 35, 44121 Ferrara, Italy}
\email{lorenzo.brasco@unife.it}

\author[Mazzoleni]{Dario Mazzoleni}
\address[D.~Mazzoleni]{Dipartimento di Matematica e Fisica N. Tartaglia
\newline\indent
Universit\`a Cattolica del Sacro Cuore
\newline\indent
Via Trieste 17, 25121 Brescia, Italy}
\email{dariocesare.mazzoleni@unicatt.it}

\begin{abstract}
We provide a sharp double-sided estimate for Poincar\'e-Sobolev constants on a convex set, 
 in terms of its inradius and $N-$dimensional measure. 
Our results extend and unify previous works by Hersch and Protter (for the first eigenvalue) and of Makai, P\'olya and Szeg\H{o} (for the torsional rigidity), by means of a single proof.
\end{abstract}

\maketitle

\begin{center}
\begin{minipage}{8cm}
\small
\tableofcontents
\end{minipage}
\end{center}

\section{Introduction}

\subsection{Principal frequencies and volume}
For every open and bounded set $\Omega\subset \mathbb{R}^N$, we start by considering two of the most studied {\it shape functionals}. Namely, its {\it principal frequency} (or  {\it first eigenvalue of the Dirichlet-Laplacian}) $\lambda(\Omega)$, and its {\it torsional rigidity} $T(\Omega)$. These are defined by 
\[
\lambda(\Omega)=\inf_{\varphi\in C^\infty_0(\Omega)\setminus\{0\}} \frac{\displaystyle\int_\Omega |\nabla \varphi|^2\,dx}{\displaystyle \int_\Omega |\varphi|^2\,dx}\quad \mbox{and}\quad T(\Omega)=\sup_{\varphi\in C^\infty_0(\Omega)\setminus\{0\}} \frac{\displaystyle \left(\int_\Omega |\varphi|\,dx\right)^2}{\displaystyle\int_\Omega |\nabla \varphi|^2\,dx}.
\]
These two quantities are well-studied and find many applications in different problems. In this paper, we will consider more generally the following {\it generalized principal frequencies}
\begin{equation}
\label{defla2q}
\lambda_{2,q}(\Omega)=\inf_{\varphi\in C^\infty_0(\Omega)\setminus\{0\}} \frac{\displaystyle\int_\Omega |\nabla \varphi|^2\,dx}{\displaystyle \left(\int_\Omega |\varphi|^q\,dx\right)^\frac{2}{q}},
\end{equation}
where $1\le q<2^*$, the latter being the usual Sobolev embedding exponent (see \cite{va, BraBus, CR} for some studies on these quantities).  Observe that the two functionals $\lambda$ and $T$ are just particular instances of this more general family of Poincar\'e-Sobolev constants, indeed with the previous notation
\[
\lambda(\Omega)=\lambda_{2,2}(\Omega)\qquad \mbox{ and }\qquad T(\Omega)=\frac{1}{\lambda_{2,1}(\Omega)}.
\]
The explicit computation of these quantities for a generic set $\Omega$ is a difficult task. Hence, it is of great importance (and interesting in itself) to provide sharp estimates for these functionals in terms of simpler quantities, often of geometric flavor. 
\vskip.2cm\noindent
The most celebrated estimate of this type is the so-called \emph{Faber-Krahn inequality}, which asserts that
\begin{equation}
\label{FK}
\lambda_{2,q}(\Omega)\ge \left(\lambda_{2,q}(B_1)\,\omega_N^{\frac{2}{N}+\frac{2-q}{q}}\right)\,|\Omega|^{-\frac{2}{N}-\frac{2-q}{q}}.
\end{equation}
Here $B_1\subset \mathbb{R}^N$ denotes a ball of unit radius and $\omega_N=|B_1|$. 
In other words, it is possible to bound from below $\lambda_{2,q}$, in terms of (a negative power of) the volume of the set.  Moreover, the estimate \eqref{FK} is sharp, since the lower bound is (uniquely) attained by balls, up to sets of zero capacity\footnote{By {\it capacity} of an open set $\Omega\subset\mathbb{R}^N$, we mean the quantity
\[
\mathrm{cap\,}(\Omega)=\inf_{u\in C^\infty_0(\mathbb{R}^N)}\left\{\int_{\mathbb{R}^N} |\nabla u|^2\,dx+\int_{\mathbb{R}^N}u^2\,dx\, :\, u\ge 1 \mbox{ on }\Omega\right\},
\]
see \cite[Chapter 4]{BB} for more details.}. We recall that \eqref{FK} follows in a standard way from the variational characterization of $\lambda_{2,q}$, by using {\it Schwarz symmetrization} and the so-called {\it P\'olya-Szeg\H{o} principle}, see \cite[Section 2]{Hen}.
\par
\noindent In passing, we observe that $\lambda_{2,q}(\Omega)$ is not comparable with $|\Omega|^{-2/N-(2-q)/q}$, not even among convex sets. For example, for a ``slab--type'' sequence, i.e. 
\[
\Omega_L=\left(-\frac{L}{2},\frac{L}{2}\right)^{N-1}\times (0,1),\qquad L>0,
\]
we have (see for example Lemma \ref{lm:slabtype} below)
\[
\lim_{L\to+\infty}\lambda_{2,q}(\Omega_L)\,|\Omega_L|^{\frac{2}{N}+\frac{2-q}{q}}=+\infty.
\]
\subsection{Principal frequencies and inradius}
In order to clarify the scope of the present paper, it is useful to recall at this point that the sharp lower bound \eqref{FK} may be quite weak for some classes of sets. 
For example, in the case $2\leq q< 2^*$, for the ``slab--type'' sequence, we have\footnote{Observe that $-\frac2N-\frac{2-q}{q}<0$ as $2\leq q<2^*$.}
\[
\lambda_{2,q}(\Omega_L)\ge \frac{1}{C} \qquad \mbox{ while }\qquad \lim_{L\to+\infty} |\Omega_L|^{-\frac{2}{N}-\frac{2-q}{q}}=0.
\]
This shows that for sets of this type, the lower bound \eqref{FK} is not very useful.
\par
In this case, a more robust and precise lower bound would be given in terms of the {\it inradius} $R_\Omega$ of a set $\Omega$, i.e. the radius of the largest open ball contained in $\Omega$. However, a {\it caveat} is needed here: such a kind of lower bound can hold true only under some suitable geometric restrictions on the sets. This is due to the fact that while a principal frequency $\lambda_{2,q}$ is not affected by removing points (and, more generally, sets with zero capacity), this operation can strongly modify $R_\Omega$: think for example of removing the center from a ball.
\par
One possibility is to work with open bounded and {\it convex} sets. By still sticking to the case $q\ge 2$ we have (see \cite[Proposition~6.3]{BraHer})
\begin{equation}
\label{HPq}
\lambda_{2,q}(\Omega) \ge \frac{C_{N,q}}{R_\Omega^{2+\frac{2-q}{q}\,N}}.
\end{equation}
As usual in this type of estimates, the power on the inradius is dictated by scale invariance.
Here $C_{N,q}>0$ is a universal constant, possibly depending on $N$ and $q$. 
\par
The value of the sharp constant in \eqref{HPq} is not known (see \cite[Open Problem 1]{BraHer}), except that for the particular case $q=2$. In this case, we know that
\begin{equation}\label{HP}
\lambda(\Omega)> \left(\frac{\pi}{2\,R_\Omega}\right)^2.
\end{equation}
We also notice that inequality in \eqref{HP} is strict among bounded convex sets, but the estimate is sharp. Indeed, for the ``slab--type'' sequence $\Omega_L$ we have
\[
\lim_{L\to+\infty}R_{\Omega_L}^2\,\lambda(\Omega_L)= \left(\frac{\pi}{2}\right)^2.
\]
Estimate \eqref{HP} has been first proved in two dimensions by Hersch in \cite{He}, by means of what he called {\it \'evaluation par d\'efaut}. The extension to higher dimensions is usually attributed to Protter, see \cite{Pr}. For this reason, we will refer to \eqref{HPq} and \eqref{HP} as {\it Hersch-Protter inequality}.
\par
In order to complete  the picture, we also recall that $\lambda_{2,q}$ is actually comparable with a power of the inradius. Indeed, by employing the monotonicity with respect to set inclusion of $\lambda_{2,q}$, we easily get
\begin{equation}
\label{banale}
\lambda_{2,q}(\Omega)\leq \frac{\lambda_{2,q}(B_1)}{R_\Omega^{2+\frac{2-q}{q}\,N}}.
\end{equation}
This inequality is optimal, as balls (uniquely) attain the equality cases. Moreover, the convexity requirement can now be dropped.
\par
For the moment, we just discussed the Hersch-Protter estimate for the case $q\ge 2$. The reason is simple: in the case $1\leq q<2$ the situation is entirely different. Indeed, as observed in \cite[Proposition 6.1]{BraHer}, {\it it is not possible to a have a Hersch-Protter estimate} in this regime. By calling again the ``slab-type'' sequence $\Omega_L$ into play, for $1\le q<2$ we have
\[
\lim_{L\to+\infty} \lambda_{2,q}(\Omega_L)=0\qquad \mbox{ and }\qquad R_{\Omega_L}=\frac{1}{2}, \mbox{ for }L>1.
\]
Thus \eqref{HPq} can not hold in this regime.
\par
On the other hand, \eqref{banale} immediately extends to this case, as well.

\subsection{Interpolating between inradius and volume}

The last observation was the starting point of the investigation pursued in the present paper. In other words, we look for suitable ``surrogates'' of the Hersch-Protter estimate \eqref{HPq}, in the case $1\le q<2$.
\par
In order to do this, we take again the example of the ``slab--type'' sequence $\Omega_L$ and analyze the asymptotic behavior of $\lambda_{2,q}(\Omega_L)$. Indeed, by Lemma \ref{lm:slabtype} below we have
\[
0<\lim_{L\to+\infty}\lambda_{2,q}(\Omega_L)\,|\Omega_L|^\frac{2-q}{q}.
\]
This suggests that a suitable Hersch-Protter estimate could hold among convex sets, provided a multiplicative correction term containing a power of the volume is taken into account. It turns out that this intuition is correct and for $1\le q<2$ we have
\[
\lambda_{2,q}(\Omega)\,|\Omega|^\frac{2-q}{q}\ge \frac{C}{R_\Omega^2}.
\]
The case $q=2$ coincides with the Hersch-Protter inequality, but curiously enough this estimate {\it does not} extend to the super-homogeneous case $2<q<2^*$. Moreover, the quantity $\lambda_{2,q}(\Omega)\,|\Omega|^{(2-q)/q}$ is actually equivalent to $R_\Omega^{-2}$.
\vskip.2cm\noindent
More precisely, the main results of this note are the following ones, whose proofs are contained in Sections \ref{sect:lower} and \ref{sect:upper} below. We refer to Section \ref{sec:prelim} for the definition of $\pi_{2,q}$.

\begin{teo}[Lower bound]
\label{teo:lowerbound}
Let $1\le q\le 2$, for every $\Omega\subset\mathbb{R}^N$ open bounded convex set, we have
\begin{equation}
\label{HPweak}
\lambda_{2,q}(\Omega)\,|\Omega|^\frac{2-q}{q}> \left(\frac{\pi_{2,q}}{2\,R_\Omega}\right)^2.
\end{equation}
The inequality is strict, but the estimate is sharp. 
\par
On the other hand, for $2<q<2^*$ we have 
\[
\inf\Big\{R^2_{\Omega}\,\lambda_{2,q}(\Omega)\,|\Omega|^\frac{2-q}{q}\, :\, \Omega\subset\mathbb{R}^N \mbox{ open bounded convex}\Big\}=0.
\]
\end{teo}

\begin{teo}[Upper bound]
\label{teo:upperbound}
Let $1\le q<2^*$, for every $\Omega\subset\mathbb{R}^N$ open bounded convex set, we have
\begin{equation}
\label{HPweakup}
\lambda_{2,q}(\Omega)\,|\Omega|^\frac{2-q}{q}\le \frac{\omega_N^\frac{2-q}{q}\,\lambda_{2,q}(B_1)}{R_\Omega^2}.
\end{equation}
The inequality is attained if and only if $\Omega$ is a ball.
Moreover, if $2\leq q<2^*$, the inequality holds among open and bounded sets without the convexity assumption.
\end{teo}

\begin{oss}[Previous results]
Our results extend to the case $q>1$ some previous results known for the case $q=1$.
Indeed, by recalling that
\[
\lambda_{2,1}(\Omega)=\frac{1}{T(\Omega)},
\]
the estimates \eqref{HPweak} and \eqref{HPweakup} can be rewritten as the double-sided control on the torsional rigidity
\begin{equation}\label{MPS}
\left(\frac{T(B_1)}{\omega_N}\right)\,|\Omega|\,R_\Omega^2\le T(\Omega)<\frac{1}{3}\,|\Omega|\,R_\Omega^2.
\end{equation}
In dimension $N=2$, the lower bound is due to\footnote{{\it Caveat} for the reader: in the notation of both \cite{PS} and \cite{Ma}, we have
$4\,T(\Omega)=P(\Omega)$.} P\'olya and Szeg\H{o} (see \cite[equation (7), page 100]{PS}), while the upper bound has been proved by Makai (see \cite[equation (3')]{Ma}). Both results have been generalized in \cite{BuGuMa} to every dimension $N\ge 2$ (see also \cite[Theorem 1.1]{DGG}). Moreover, both of them are sharp, as the lower bound is (uniquely) attained by balls, while the upper bound is asymptotically attained by the ``slab-type'' sequence $\Omega_L$.
\par
In any case, we point out that our identification of equality cases in \eqref{HPweakup} appears to be new, even for the torsional rigidity, i.e. for the case $q=1$.
\end{oss}
A comment on our proofs is in order.
\begin{oss}[Method of proof]
Our proof of \eqref{HPweak} is different from the one by Makai, dealing with the case $q=1$. The latter seems quite difficult to adapt to the case $1<q<2$. Rather, we adapt a PDE-based technique used by Kajikiya in \cite{Ka}, to give a different proof of the Hersch-Protter estimate \eqref{HP}. We show that this technique is flexible enough to be adapted to the case $1\le q<2$, without loss of sharpness. This permits to unify the results of Makai and Hersch \& Protter, by means of a single proof. 
\par
For \eqref{HPweakup} we use the very same method of proof given by P\'olya and Szeg\H{o} for the case $q=1$. This is based on a variant of the so called {\it method of interior parallels}. This consists in choosing a suitable test function in the variational formulation \eqref{defla2q}: it turns out that a function of the {\it Minkowski functional} of $\Omega$ does the job. 
While in \cite{PS} the explicit form of the extremal of the ball is used, here we show that the knowledge of this explicit form is irrelevant. All that is needed is just that there exists an extremal function for $\lambda_{2,q}(B_1)$ which is radial. We also pay particular attention to the identification of equality cases, which is a bit subtle.
\end{oss}
We now comment on the convexity assumption.
\begin{oss}[Convexity matters]
When $1\le q\le 2$, both inequality \eqref{HPweak} and inequality \eqref{HPweakup} {\it can not hold for general open sets}. 
\par
For the first one, we use again that removing points affects the inradius, but not a generalized principal frequency.
That is, by taking the sequence of bounded open sets
\[
\Omega_n=(-n,n)^N\setminus \{x=(x_1,\dots,x_N)\in \mathbb{Z}^N\, :\, |x_i|\le n-1 \mbox{ for } i=1,\dots,N \},
\]
by scaling and using that points have zero capacity, we get
\[
\lambda_{2,q}(\Omega_n)\,|\Omega_n|^\frac{2-q}{q}=\lambda_{2,q}\Big((-n,n)^N\Big)\,(2\,n)^{N\,\frac{2-q}{q}}=\frac{2^{N\,\frac{2-q}{q}}\,\lambda_{2,q}\Big((-1,1)^N\Big)}{n^2}.
\]
This implies that 
\[
\lim_{n\to\infty}\lambda_{2,q}(\Omega_n)\,|\Omega_n|^\frac{2-q}{q}=0,
\]
while it is easily seen that $R_{\Omega_n}=\sqrt{N}/2$.
\par
As a counterexample to \eqref{HPweak}, one can consider a disjoint union of balls
\[
\Omega_n=\bigcup_{i=1}^n B_{r_i}(x_i),
\]
with the radius given by
\[
r_i=\sqrt[N]{\frac{1}{i}},\qquad i\ge 1,
\]
and the centers of the balls chosen so that $B_{r_i}(x_i)\cap B_{r_j}(x_j)=\emptyset$ for all $i\not= j$.
This choice guarantees that
\[
R_{\Omega_n}=r_1=1\qquad  \mbox{ and }\qquad \lim_{n\to\infty}|\Omega_n|=+\infty,
\]
while by \cite[Example 5.2]{BR} we have
\[
\lim_{n\to\infty} \lambda_{2,q}(\Omega_n)>0.
\]
We point out that this is no more a counterexample as soon as $q\geq 2$, as the exponent of the measure term $(2-q)/q$ becomes non-positive.
\end{oss}

\subsection{Plan of the paper}
After the Introduction, in Section~\ref{sec:prelim} we fix the notation, recall some known facts about the Poincar\'e-Sobolev constants and give some properties of the Minkowski functional of a convex set.
Section~\ref{sect:lower} is devoted to the proof of Theorem \ref{teo:lowerbound}, while in Section~\ref{sect:upper} we prove the upper bound of Theorem \ref{teo:upperbound}.
Finally, in Section~\ref{sect:further} we discuss the case of more general versions of our ``mixed'' estimate, in terms of different powers of volume and inradius.

\begin{ack}
The initial input for this research has been a question raised by Andrea Malchiodi during a talk of the first author. We wish to thank him. We also thank Vladimir Bobkov, for pointing out the paper \cite{Ka} to our attention.
\par
D.\,M. has been supported by the INdAM-GNAMPA 2019 project  ``{\it Ottimizzazione spettrale non lineare\,}''.
Part of this work has been done during a visit of L.\,B. to Brescia and a visit of D.\,M. to Ferrara. The hosting institutions and their facilities are gratefully acknowledged.
\end{ack}

\section{Preliminaries}\label{sec:prelim}
\subsection{Notation}
For the whole paper, $N\geq 2$ is the dimension of the space and we denote by $2^*$ the critical Sobolev exponent, i.e.
\[
2^*=\left\{\begin{array}{cc}
\dfrac{2\,N}{N-2},& \mbox{ if } N\ge 3,\\
&\\
+\infty, & \mbox{ if } N=2.
\end{array}
\right.
\] 
For an open set $\Omega\subset \mathbb{R}^N$, we denote by $|\Omega|$ its $N-$dimensional Lebesgue measure and use the standard notation for the balls:\[
B_R(x_0)=\left\{x\in\mathbb{R}^N : |x-x_0|<R\right\},\qquad \omega_N=|B_1(0)|.
\]
We will omit the center when this will coincide with the origin.
Whenever it is well-defined, we call $\nu_\Omega(x)$ the outer unit normal versor at a point $x\in\partial \Omega$.
\subsection{Inradius}
\label{subsec:inradius}
For an open bounded set $\Omega\subset \mathbb{R}^N$ with Lipschitz boundary, we define the distance function from the boundary
\[
d_{\Omega}(x)=\inf_{y\in \partial\Omega}|x-y|,\qquad \mbox{ for }x\in\Omega.
\]
We recall that this is a $1-$Lipschitz function. Moreover, if $\Omega$ is convex, then $d_\Omega$ is concave and thus it is a weakly superharmonic function.
It is well-known that the {\it inradius} $R_\Omega$ of $\Omega$ (i.e. the radius of the largest ball included in $\Omega$) coincides with 
\[
R_\Omega=\sup_{x\in\Omega} d_\Omega(x).
\]
We present now a property of convex sets related to the inradius, which we will use in the proof of the rigidity for the upper bound  \eqref{HPweakup}. Though it should be somehow classical, we did not find a precise reference, so we give a proof for completeness.
\begin{lm}
\label{lm:inradius}
Let $\Omega\subset\mathbb{R}^N$ be an open bounded convex set. Let us suppose that for some $R>0$, $B_R\subset\Omega$, then we have
\begin{equation}
\label{inradiusgen}
R\le \langle x,\nu_\Omega(x)\rangle,\qquad \mbox{ for $\mathcal{H}^{N-1}-$a.\,e. }x\in\partial\Omega.
\end{equation}
Moreover, if $\Omega$ is of class $C^1$ and we have
\[
R= \langle x,\nu_\Omega(x)\rangle,\qquad \mbox{ for every }x\in\partial\Omega,
\]
then it must hold $\Omega=B_R$.
\end{lm}
\begin{proof}
We observe that for every $x\in\partial \Omega$, by convexity
\[
\overline{B_{R}}\subset \overline{\Omega} \subset \Big\{y\in\mathbb{R}^N\, :\, \langle y-x,\nu_\Omega(x)\rangle\le 0 \Big\}.
\]
In particular, by taking the point $y=R\,\nu_\Omega(x)\in\partial B_{R}$, we get
\[
R=\langle R\,\nu_\Omega(x),\nu_\Omega(x)\rangle\le \langle x,\nu_\Omega(x)\rangle.
\]
This proves \eqref{inradiusgen}.
\vskip.2cm\noindent
We now suppose that $\Omega$ is of class $C^1$ and assume that equality in \eqref{inradiusgen} holds for every $x\in\partial\Omega$. We argue by contradiction and suppose that $\Omega\not= B_R$. We take $x_0\in\partial\Omega$ such that
\[
T:=\mathrm{dist}(x_0,\partial B_R)=\max_{y\in\partial\Omega} \mathrm{dist}(y,\partial B_R).
\]
By the contradiction assumption, we have $T>0$.
Up to a rigid movement, we can suppose that \[
x_0=(0,\dots,0,-T-R),
\] 
and find a certain $r_0>0$ such that 
\[
\Gamma:=\partial\Omega\cap \Big([-r_0,r_0]^{N-1}\times[-T-R-r_0,-T-R+r_0]\Big),
\] 
coincides with the graph of a $C^1$ convex function $\Phi:[-r_0,r_0]^{N-1}\to \mathbb{R}$, with
 \[
\Phi(0,\dots, 0)=-T-R<-R.
\]
Moreover, by maximality of $x_0$, we have 
\begin{equation}
\label{nabla0}
\nabla \Phi(0,\dots,0)=0,
\end{equation}
see Figure \ref{fig:rigid}.
\begin{figure}
\includegraphics[scale=.3]{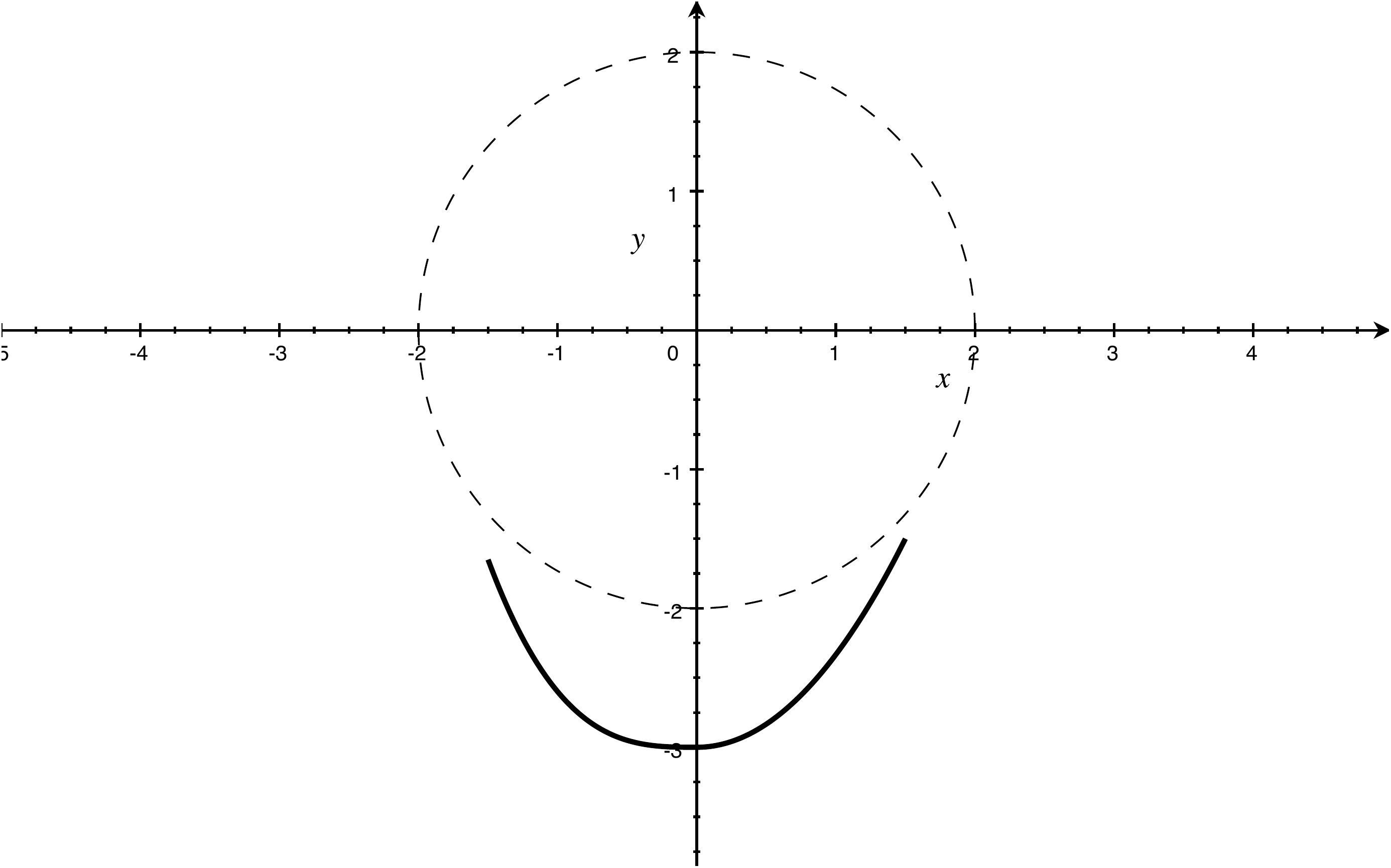}
\caption{The geometric configuration in the proof of Lemma \ref{lm:inradius}. The bold line represents the graph of $\Phi$.}
\label{fig:rigid}
\end{figure}
Then, at the point $x_0=(0,\dots,0,-T-R)$ we have $\nu_\Omega(x_0)=(0,\dots,0,-1)$. Thus, by using that we have equality in \eqref{inradiusgen}, we get
\[
R=\langle x_0,\nu_\Omega(x_0)\rangle=T+R.
\]
Since $T>0$, this gives the desired contradiction.
\end{proof}
\begin{oss}[The importance of being $C^1$]
The $C^1$ assumption is crucial to get the condition \eqref{nabla0}.
On the other hand, when $\Omega$ is convex but not $C^1$, then {\it it is no more true} that 
\[
``R= \langle x,\nu_\Omega(x)\rangle,\quad \mbox{ for $\mathcal{H}^{N-1}-$a.\,e.  }x\in\partial\Omega''\qquad \Longrightarrow \qquad \Omega=B_R.
\]
In fact, there are lots of convex sets for which this identity holds true. For example, it is sufficient to take any convex polyhedron, such that each of its faces touches the ball $B_R$. Another example can be found by taking the cone obtained as the convex envelope of $B_R$ and a point $x_0\in\mathbb{R}^N\setminus B_R$, see Figure \ref{fig:2}.
\begin{figure}[h]
\includegraphics[scale=.3]{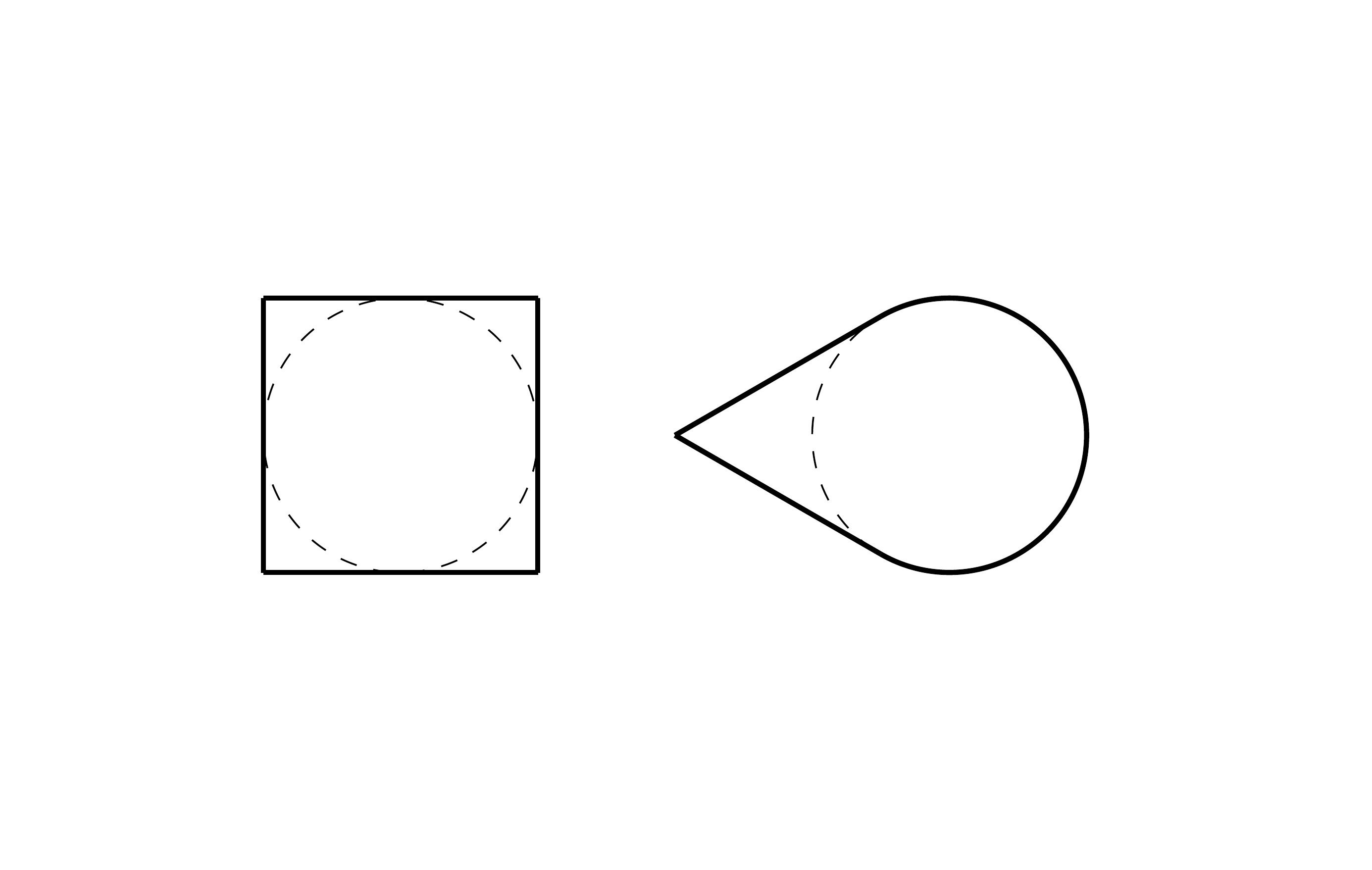}
\caption{Two convex sets for which equality in \eqref{inradiusgen} holds almost everywhere on the boundary.}
\label{fig:2}
\end{figure}
\end{oss}
\subsection{Poincar\'e-Sobolev constants}
For $1\leq q<2^*$ and every open set $\Omega\subset \mathbb{R}^N$, we have defined in \eqref{defla2q} the Poincar\'e-Sobolev constants, which we can interpret as generalizations of the first eigenvalue of the Dirichlet-Laplacian. They can be equivalently characterized as 
\[
\lambda_{2,q}(\Omega)=\inf_{\varphi\in C^\infty_0(\Omega)}\left\{\int_\Omega|\nabla \varphi|^2\,dx\, :\, \|\varphi\|_{L^q(\Omega)}=1\right\}.
\]
If $\Omega$ is bounded or, more generally, has finite measure, the infimum is attained on the homogeneous Sobolev space $\mathcal D^{1,2}_0(\Omega)$, defined as the completion of $C^\infty_0(\Omega)$ with respect to the norm 
\[
u\mapsto \|\nabla u\|_{L^2(\Omega)}.
\]
In this case, a minimizer $u\in\mathcal{D}^{1,2}_0(\Omega)$ of the previous problem weakly solves the Lane-Emden equation
\[
-\Delta u=\lambda_{2,q}(\Omega)\,|u|^{q-2}\,u,\qquad \mbox{ in }\Omega.
\]
As we already recalled in the introduction, the quantity $1/\lambda_{2,1}(\Omega)$ coincides with the torsional rigidity $T(\Omega)$.
\par
From the definition, it is easy to check that these quantities scale as 
\[
\lambda_{2,q}(t\,\Omega)=t^{-2-\frac{2-q}{q} N}\lambda_{2,q}(\Omega),\qquad t>0.
\]
Finally, as it is clear from the statement of the Theorem \ref{teo:lowerbound}, the constants $\pi_{2,q}$ play a fundamental role in our work. These are nothing but the one-dimensional Poincar\'e-Sobolev constants, more precisely they are defined by
\[
\pi_{2,q}=\min_{\varphi\in W^{1,2}((0,1))\setminus\{0\}}{\left\{\frac{\|\varphi'\|_{L^2((0,1))}}{\|\varphi\|_{L^q((0,1))}}\, :\, \varphi(0)=\varphi(1)=0\right\}}.
\]
We refer to~\cite[Appendix~A]{BraBus} and \cite[Section 5]{FL} for more details.
It is worth recalling some explicit values for these constants, see~\cite[Remark~2.4]{BraBus},
\[
\pi_{2,1}=2\,\sqrt{3}\qquad \mbox{ and }\qquad \pi_{2,2}=\pi.
\]
Then it is immediate to see that Makai's upper bound in~\eqref{MPS} coincides the lower bound in~\eqref{HPweak} when $q=1$, while the Hersch-Protter estimate~\eqref{HP} is contained again in~\eqref{HPweak} when $q=2$.

The relation between the constants $\pi_{2,q}$ and  $\lambda_{2,q}$ for the ``slab--type'' sequence $\Omega_L$ is detailed in the Appendix, see Lemma~\ref{lm:slabtype}.
\subsection{The Minkowski functional}
Here we recall the definition and main properties of the {\it Minkowski functional} of a convex set $\Omega\subset \mathbb{R}^N$ such that $0\in\Omega$, denoted by $j_\Omega$. This is defined by
\[
j_\Omega(x):=\inf\Big\{r>0 : x\in r\,\Omega\Big\}.
\]
First of all, by construction it is easily seen that
\begin{equation}
\label{levelset}
\{x\in\mathbb{R}^N\, :\, j_\Omega(x)=t\}=t\,(\partial\Omega),\qquad \mbox{ for every }t>0.
\end{equation}
The main properties of $j_\Omega$ needed for our purposes are summarized in the following
\begin{lm}
\label{lm:mink}
Let $\Omega\subset\mathbb{R}^N$ be an open bounded convex set, such that $0\in\Omega$. The function $j_\Omega$ is a convex Lipschitz and  positively $1$-homogeneous function, i.e.
\[
j_\Omega(t\,x)=t\,j_\Omega(x),\qquad \mbox{ for every } x\in\mathbb{R}^N,\ t>0.
\]
Moreover, $j_\Omega$ is differentiable for $\mathcal{H}^{N-1}-$almost every $x\in \partial\Omega$ and it holds
\begin{equation}
\label{scalar}
\langle x,\nu_\Omega(x)\rangle=\frac{1}{|\nabla j_\Omega(x)|},\qquad \mbox{ for $\mathcal{H}^{N-1}-$a.\,e. } x\in \partial \Omega.
\end{equation}
\end{lm}
\begin{proof}
For completeness, we sketch the proof of these classical facts.
The homogeneity of $j_\Omega$ is a straightforward consequence of its definition. Moreover, by still using its definition, it is not difficult to see that $j_\Omega$ is level convex, i.e.
\[
j_\Omega((1-t)\,z+t\,w)\le \max\{j_\Omega(z),\, j_\Omega(w)\},\qquad \mbox{ for every } z,w\in\mathbb{R}^N, t\in[0,1].
\]
By using this property with  
\[
z=\frac{x}{j_\Omega(x)},\quad w=\frac{y}{j_\Omega(y)},\quad t=\frac{j_\Omega(y)}{j_\Omega(x)+j_\Omega(y)}
\]
and using the positive $1-$homogeneity, we then get that $j_\Omega$ is sub-additive. Finally, from this we obtain
\[
j_\Omega((1-t)\,x+t\,y)\le j_\Omega((1-t)\,x)+j_\Omega(t\,y)=(1-t)\,j_\Omega(x)+t\,j_\Omega(y),
\] 
i.e. $j_\Omega$ is convex.
\par
As a convex function, it is automatically locally Lipschitz. By positive $1-$homogeneity, we can upgrade this information to a global Lipschitz continuity. In any case, we have that $j_\Omega$ is differentiable almost everywhere in $\mathbb{R}^N$.
\par
To prove that $j_\Omega$ is differentiable almost everywhere on $\partial\Omega$, we use again the positive $1$-homogeneity: indeed, if there exists $\Sigma\subset\partial\Omega$ such that
\[
\mathcal{H}^{N-1}(\Sigma)>0,
\]
and $j_\Omega$ is not differentiable on $\Sigma$, then $j_\Omega$ would automatically be not differentiable on the cone generated by $\Sigma$, i.e.
\[
C_\Sigma=\{x\in\Omega\, :\, x=t\,y \mbox{ for some } t\in[0,1],\ y\in\Sigma \}.
\]
But this would be a set with positive $N-$dimensional measure, on which $j_\Omega$ is not differentiable, thus giving a contradiction.
\par 
Finally, by differentiating in $t$ the identity
\[
j_\Omega(t\,x)=t\,j_\Omega(x),
\]
and taking $t=1$, we get
\[
\langle\nabla j_\Omega(x),x\rangle=j_\Omega(x).
\]
By recalling \eqref{levelset}, the latter implies 
\[
|\nabla j_\Omega(x)|\,\langle \nu_\Omega(x),x\rangle=1,\qquad \mbox{ for $\mathcal{H}^{N-1}-$a.\,e. } x\in\partial\Omega.
\]
This concludes the proof.
\end{proof}

\section{Lower bound}\label{sect:lower}

In order to prove Theorem \ref{teo:lowerbound}, the following technical result will be useful. The proof is standard, we give it for completeness.
\begin{lm}
\label{lm:composizione}
Let $q\geq 1$ and $f\in C^2([a,b])$ be a non-decreasing function, such that
\[
-f''=C\,f^{q-1},\ \mbox{ in } [a,b],\qquad \mbox{ with } f(a)=0.
\]
Let $\Omega\subset\mathbb{R}^N$ be an open set and let $u\in W^{1,2}(\Omega)\cap L^\infty(\Omega)$ be a weakly superhamonic function, i.e.
\begin{equation}
\label{superharmonic}
\int_\Omega \langle \nabla u,\nabla \varphi\rangle\,dx\ge 0,\qquad \mbox{ for every } \varphi\in C^\infty_0(\Omega) \mbox{ with }\varphi\ge 0.
\end{equation}
Let us assume in addition that 
\[
a\le u(x)\le b,\qquad \mbox{ for a.\,e. }x\in\Omega.
\]
Then the composition $\phi=f\circ u$ satisfies $-\Delta \phi\ge C\,\phi^{q-1}\,|\nabla u|^2$
in weak sense, i.\,e.
\[
\int_\Omega \langle \nabla \phi,\nabla \varphi\rangle\,dx\ge C\,\int_\Omega \phi^{q-1}\,|\nabla u|^2\,\varphi\,dx,
\]
for every $\varphi\in C^\infty_0(\Omega)$ with $\varphi\ge 0$.
\end{lm}
\begin{proof}
We take $\eta\in C^\infty_0(\Omega)$ non-negative and insert in \eqref{superharmonic} the test function\footnote{Observe that this $\varphi$ is only a $W^{1,2}$ function with compact support in $\Omega$, but by a standard density argument it is clearly admissible.}
\[
\varphi=f'(u)\,\eta\ge 0.
\] 
We obtain
\[
\begin{split}
0\le \int_\Omega \langle \nabla u,\nabla (f'(u)\,\eta)\rangle\,dx&=\int_\Omega |\nabla u|^2\,f''(u)\,\eta\,dx\\
&+\int_\Omega \langle f'(u)\,\nabla u,\nabla \eta\rangle\,dx\\
&=-C\,\int_\Omega |\nabla u|^2\,f(u)^{q-1}\,\eta\,dx\\
&+\int_\Omega \langle \nabla f(u),\nabla\eta\rangle\,dx.
\end{split}
\]
By recalling the definition of $\phi=f\circ u$, this gives the desired result. 
\end{proof}
We are now in position to prove Theorem \ref{teo:lowerbound}.
\begin{proof}[Proof of Theorem \ref{teo:lowerbound}]
We divide the proof in three parts.
\vskip.2cm\noindent
{\bf 1. Inequality for $1\le q\le 2$.} We adapt the trick of \cite{Ka} for proving the Hersch-Protter inequality \eqref{HP}, i.\,e. for the case $q=2$.
We take $v\in W^{1,2}((-1,0))$ to be a positive solution of
\begin{equation}
\label{min1d}
\min_{\varphi\in W^{1,2}((-1,0))}\left\{\int_{-1}^0 |\varphi'|^2\,dt\, :\, \int_{-1}^0 |\varphi|^q\,dt=1 \ \mbox{and }\  \varphi(-1)=0\right\}.
\end{equation}
We can assume that $v$ is non-decreasing: indeed, if this were not the case, we could consider the new function
\[
w(t)=\int_{-1}^t |v'(\tau)|\,d\tau,
\]
which is positive and non-decreasing by construction and such that $w(-1)=0$. Moreover, we have 
\[
\int_{-1}^0 |w'(t)|^2\,dt=\int_{-1}^0 |v'(t)|^2\,dt,
\]
and
\[
w(t)^q=\left(\int_{-1}^t |v'(\tau)|\,d\tau\right)^q\ge \left|\int_{-1}^t v'(\tau)\,d\tau\right|^q=|v(t)|^q,
\]
so that 
\[
\int_{-1}^0 w^q\,dt\ge 1.
\]
This implies that $\widetilde{w}=w\,\|w\|_{L^q((-1,0))}^{-1}$ is another positive minimizer of \eqref{min1d}, and we can work with it.
\par
By recalling that the minimal value \eqref{min1d} coincides with $(\pi_{2,q}/2)^2$ (see \cite[Lemma A.1]{BraBus}), we get that $v$ is a positive and non-decreasing function that solves the following mixed problem
\[
\left\{\begin{array}{rllc}
-v''&=&\left(\dfrac{\pi_{2,q}}{2}\right)^2\,v^{q-1},& \mbox{ in }(-1,0)\\
&&&\\
v(-1)&=&v'(0)=0.
\end{array}
\right.
\]
Actually, by using the equation it is easily seen that $v$ is of class $C^2$ on the interval $[-1,0]$.
We then consider the function (recall the notation of Subsection \ref{subsec:inradius})
\[
\phi(x)=v\left(\frac{d_\Omega(x)}{R_\Omega}-1\right),
\]
and observe that 
\begin{equation}
\label{eikonal}
|\nabla d_\Omega|^2=1,\qquad \mbox{ a.\,e. in }\Omega,
\end{equation}
and that $d_\Omega$ is weakly superharmonic in $\Omega$, thanks to the convexity of the set.
We can then apply Lemma \ref{lm:composizione} with the choices
\[
f=v\qquad \mbox{ and }\qquad u=\frac{d_\Omega}{R_\Omega}-1,
\] 
and obtain that
\[
\left(\frac{\pi_{2,q}}{2\,R_\Omega}\right)^2\,\int_\Omega \phi^{q-1}\,\varphi\,dx\le \int_\Omega \langle \nabla \phi,\nabla\varphi\rangle\,dx,\quad \mbox{ for every } \varphi\in C^\infty_0(\Omega),\, \varphi\ge 0.
\]
By a standard density argument, in the previous equation we can also admit test functions in $W^{1,2}$ with compact support in $\Omega$.
In particular, by taking $\varphi=\eta^2/\phi$ 
with $\eta\in C^\infty_0(\Omega)$ and using Picone's inequality, we get
\[
\begin{split}
\left(\frac{\pi_{2,q}}{2\,R_\Omega}\right)^2\,\int_\Omega \phi^{q-2}\,\eta^2\,dx&\le \int_\Omega \left\langle \nabla \phi,\nabla\left(\frac{\eta^2}{\phi}\right)\right\rangle\,dx\le \int_\Omega |\nabla\eta|^2\,dx.
\end{split}
\]
This in particular implies that
\begin{equation}
\label{pre}
\left(\frac{\pi_{2,q}}{2\,R_\Omega}\right)^2\le \frac{\displaystyle \int_\Omega |\nabla\eta|^2\,dx}{\displaystyle \int_\Omega \phi^{q-2}\,\eta^2\,dx}.
\end{equation}
We now observe that by H\"older's inequality we have
\[
\int_\Omega \phi^{q-2}\,\eta^2\,dx\ge \frac{\displaystyle\left(\int_\Omega \eta^q\,dx\right)^\frac{2}{q}}{\displaystyle\left(\int_\Omega \phi^q\,dx\right)^\frac{2-q}{q}}.
\]
By using this in \eqref{pre} and then taking the infimum over $\eta$, we obtain
\begin{equation}
\label{pre2}
\left(\frac{\pi_{2,q}}{2\,R_\Omega}\right)^2\le \lambda_{2,q}(\Omega)\,\left(\int_\Omega \phi^q\,dx\right)^\frac{2-q}{q}.
\end{equation}
If $q=2$ the proof is over and we obtain the Hersch-Protter estimate. If $1\le q<2$, we are now left with estimating from above the $L^q$ norm of $\phi$.
By using the definition of $\phi$, the Coarea Formula and property \eqref{eikonal}, we get
\[
\begin{split}
\int_\Omega \phi^q\,dx&=\int_0^{R_\Omega} \left(v\left(\frac{t}{R_\Omega}-1\right)\right)^q\,P(\Omega_t)\,dt,
\end{split}
\]
where we set
\[
\Omega_t=\{x\in\Omega\, :\, d_\Omega(x)>t\},
\]
and by $P(\Omega_t)$ we denote the perimeter of the set $\Omega_t$.
We recall that the perimeter is monotone increasing with respect to set inclusion, in the class of convex sets (see \cite[Lemma 2.2.2]{BB}). Since $\Omega_t$ is convex, we get that the function $\psi(t)=P(\Omega_t)$ is monotone decreasing. Moreover, if we define 
\[
\xi(t)=\int_0^t \left(v\left(\frac{\tau}{R_\Omega}-1\right)\right)^q\,d\tau,\qquad \mbox{ for } t\in [0,R_\Omega],
\]
then it is easily seen that the pair $(\xi,\psi)$ verifies the assumptions of Lemma \ref{lm:tecnico}. Indeed, by using that $v$ is monotone non-decreasing on $[-1,0]$, we can infer that
\[
\xi(t)\le t\,\left(v\left(\frac{t}{R_\Omega}-1\right)\right)^q=t\,\xi'(t),
\]
which entails that $t\mapsto \xi(t)/t$ is increasing. By applying Lemma \ref{lm:tecnico}, we thus get
\[
\begin{split}
\int_\Omega \phi^q\,dx&=\int_0^{R_\Omega} \left(v\left(\frac{t}{R_\Omega}-1\right)\right)^q\,P(\Omega_t)\,dt\\
&\le \frac{\xi(R_\Omega)}{R_\Omega}\,\int_0^{R_\Omega} P(\Omega_t)\,dt=\frac{\xi(R_\Omega)}{R_\Omega}\,|\Omega|.
\end{split}
\]
It is only left to observe that by definition of $\xi$, with a simple change of variable we have
\[
\frac{\xi(R_\Omega)}{R_\Omega}=\frac{1}{R_\Omega}\,\int_0^{R_\Omega} \left(v\left(\frac{\tau}{R_\Omega}-1\right)\right)^q\,d\tau=\int_{-1}^0 v^q\,ds=1,
\]
where we have used that the function $v$ has unit $L^q$ norm on the interval $[-1,0]$.
This shows that
\[
\int_\Omega \phi^q\,dx\leq |\Omega|.
\]
Finally, by spending this information into \eqref{pre2}, we get the desired estimate.
\vskip.2cm\noindent
{\bf 2. Sharpness for $1\le q\le 2$.} To prove the sharpness we consider the ``slab--type'' sequence 
\[
\Omega_L=\left(-\frac{L}{2},\frac{L}{2}\right)^{N-1}\times (0,1).
\]
By Lemma \ref{lm:slabtype} below, we know that
\[
\lambda_{2,q}(\Omega_L)\sim \frac{\Big(\pi_{2,q}\Big)^2}{L^{(N-1)\,\frac{2-q}{q}}},\qquad \mbox{ as } L\to +\infty.
\]
On the other hand, by construction it is easy to check that
\[
R_{\Omega_L}=\frac{1}{2}\qquad \mbox{and }\qquad  |\Omega_L|\sim L^{N-1}, \qquad \mbox{ as } L\to +\infty.
\]
Finally, we obtain
\[
 \lambda_{2,q}(\Omega_L)\,|\Omega_L|^\frac{2-q}{q}\sim \left(\frac{\pi_{2,q}}{2\,R_{\Omega_L}}\right)^2,\qquad \mbox{ as }L\to+\infty,
\]
so we have proved the sharpness of the estimate \eqref{HPweak}.
\vskip.2cm\noindent
{\bf 3. The case $2<q<2^*$.} We still take the family of sets $\Omega_L$ as above. In this case, by Lemma \ref{lm:slabtype} we have 
\[
\lambda_{2,q}(\Omega_L)\sim \lambda_{2,q}(\mathbb{R}^{N-1}\times(0,1))>0,\qquad \mbox{ as } L\to +\infty.
\]
Thus, we now get
\[
R_{\Omega_L}^2\,\lambda_{2,q}(\Omega_L)\,|\Omega_L|^\frac{2-q}{q}\sim \frac{1}{4}\, \lambda_{2,q}(\mathbb{R}^{N-1}\times(0,1))\,L^{(N-1)\,\frac{2-q}{q}},\qquad \mbox{ as } L\to +\infty,
\]
and this quantity converges to $0$, thanks to the fact that $2-q<0$. This concludes the proof.
\end{proof}

\section{Upper bound}\label{sect:upper}

The proof of the upper bound is based on the use of a clever test function. The idea is quite similar to the so-called {\it method of interior parallels}. The latter uses test functions of the form
\[
u(x)=\varphi(d_\Omega(x)).
\]
In our case, on the contrary, test functions of the form 
\[
u(x)=\varphi(j_{\Omega}(x)),
\] 
will do the job. As recalled in Section \ref{sec:prelim}, $j_\Omega$ is the Minkowski functional of a convex set $\Omega\subset \mathbb{R}^N$ such that $0\in\Omega$. Its relevant properties needed in the following proof are contained in Lemma \ref{lm:mink}.
\begin{proof}[Proof of Theorem \ref{teo:upperbound}]
We divide the proof in two cases, depending on whether $q<2$ or $q\ge 2$.
\vskip.2cm\noindent
{\bf Case $2\leq q <2^*$.} In this case, the proof is trivial. It is sufficient to observe that, for all $\Omega\subset\mathbb{R}^N$ open and bounded, we have
\[
R_\Omega^2\,\lambda_{2,q}(\Omega)\,|\Omega|^\frac{2-q}{q}=\left(\lambda_{2,q}(\Omega)\,R_\Omega^{2-\frac{2-q}{q}\,N}\right)\,\left(\frac{|\Omega|}{R_\Omega^N}\right)^\frac{2-q}{q},
\]
and use that both quantities are (uniquely) maximized by balls among open sets. In fact, by monotonicity of $\lambda_{2,q}$ with respect to set inclusion, we have 
\[
\lambda_{2,q}(\Omega)\leq \lambda_{2,q}(B_{R_\Omega})=\lambda_{2,q}(B_1)\,R_\Omega^{2-\frac{2-q}{q}\,N}.
\]
On the other hand, since $\Omega$ contains a ball of volume $\omega_N\,R_\Omega^N$, it is clear that 
\[
\frac{R_\Omega^N}{|\Omega|}\leq \frac{1}{\omega_N}.
\]
By recalling that $2-q\le 0$, we get the conclusion.
\vskip.2cm\noindent
{\bf Case $1\leq q< 2$.} By definition of inradius, we have that $\Omega$ contains a ball of radius $R_\Omega$. Without loss of generality, we can assume that such a ball is centered at the origin. 
\par
We take $u\in \mathcal D^{1,2}_0(B_1)$ to be optimal for the variational problem defining $\lambda_{2,q}(B_1)$. Without loss of generality, we can take $u$ to be positive. Moreover, we know that it must be a radially symmetric function (see \cite[Theorem 3]{Kaw}). Thus, there exists a $C^1$ function $f:[0,1]\rightarrow [0,+\infty)$ such that 
\[
u(x)=f(|x|), \qquad \mbox{ for every }x\in B_1.
\]
The previous properties of $u$ entail that $f$ is decreasing and that $f'(0)=0$.
By using spherical coordinates, we have 
\begin{equation}
\label{sfera}
\lambda_{2,q}(B_1)=\frac{N\,\omega_N\,\displaystyle \int_0^1|f'(t)|^2\,t^{N-1}\,dt}{\left(N\,\omega_N\,\displaystyle\int_0^1f(t)^q\,t^{N-1}\,dt\right)^{\frac{2}{q}}}.
\end{equation}
We then use the composition $f\circ j_\Omega$ as a test function in the Rayleigh quotient defining $\lambda_{2,q}(\Omega)$. Indeed, $f$ is $C^1$, $j_\Omega$ is Lipschitz and observe that we have 
\[
f(j_\Omega(x))=f(1)=0,\qquad \mbox{ for every }x\in\partial\Omega,
\] 
so that $f\circ j_\Omega\in \mathcal D^{1,2}_0(\Omega)$.
We first compute the $L^q$ norm of this test function. By using the Coarea Formula and the property \eqref{levelset}, we have
\[
\int_\Omega f(j_\Omega)^q\,dx=\int_0^1 f(t)^q\,\left(\int_{t\,\partial\Omega}\frac{1}{|\nabla j_\Omega(x)|}\,d\mathcal H^{N-1}\right)\,dt.
\]
By using the change of variable $x=t\,y$ and the fact that $\nabla j_\Omega$ is positively $0-$homogeneous, we get
\[
\int_\Omega f(j_\Omega)^q\,dx=\left(\int_0^1f(t)^q\,t^{N-1}\,dt\right)\,\left(\int_{\partial \Omega}\frac{1}{|\nabla j_\Omega(x)|}\,d\mathcal H^{N-1}\right).
\]
If we further use \eqref{scalar} and the Divergence Theorem, we finally get
\begin{equation}
\label{denominatore}
\begin{split}
\int_\Omega f(j_\Omega)^q\,dx&=\left(\int_0^1f(t)^q\,t^{N-1}\,dt\right)\,\left(\int_{\partial \Omega}\langle x,\nu_\Omega\rangle \,d\mathcal H^{N-1}\right)\\
&=N\,|\Omega|\,\int_0^1\,f(t)^q\,t^{N-1}\,dt.
\end{split}
\end{equation}
%
We proceed similarly, in order to estimate the Dirichlet integral. By using again the Coarea Formula, \eqref{levelset} and the change of variable $x=t\,y$ as above, we have 
\[
\int_\Omega |\nabla f(j_\Omega)|^2\,dx=\left(\int_0^1 |f'(t)|^2\,t^{N-1}\,dt\right)\,\left(\int_{\partial\Omega}|\nabla j_\Omega|\,d\mathcal H^{N-1}\right).
\]
We use again \eqref{scalar}, so to obtain
\begin{equation}
\label{lunedi}
\int_\Omega |\nabla f(j_\Omega)|^2\,dx=\left(\int_0^1 |f'(t)|^2\,t^{N-1}\,dt\right)\,\left(\int_{\partial\Omega}\frac{1}{\langle x,\nu_\Omega\rangle}\,d\mathcal H^{N-1}\right).
\end{equation}
We now estimate the scalar product. By Lemma \ref{lm:inradius}, we have
\begin{equation}
\label{inradius}
R_\Omega\le \langle x,\nu_\Omega(x)\rangle,\qquad \mbox{ for $\mathcal{H}^{N-1}-$a.\,e. }x\in\partial\Omega.
\end{equation}
Thus we get the following lower bound
\[
\langle x,\nu_\Omega\rangle=\frac{\langle x,\nu_\Omega\rangle^2}{\langle x,\nu_\Omega\rangle}\ge \frac{R_\Omega^2}{\langle x,\nu_\Omega\rangle}, \qquad \mbox{ for $\mathcal{H}^{N-1}-$a.\,e.  }x\in\partial\Omega.
\]
By inserting this into \eqref{lunedi}, we get
\begin{equation}
\label{numeratore}
\begin{split}
\int_\Omega |\nabla f(j_\Omega)|^2\,dx&\le \frac{1}{R_\Omega^2}\,\left(\int_0^1 |f'(t)|^2\,t^{N-1}\,dt\right)\,\left(\int_{\partial\Omega}\langle x,\nu_\Omega\rangle\,d\mathcal H^{N-1}\right)\\
&=\frac{N\,|\Omega|}{R_\Omega^2}\,\left(\int_0^1 |f'(t)|^2\,t^{N-1}\,dt\right).
\end{split}
\end{equation}
Now, by putting together \eqref{numeratore} and \eqref{denominatore}, we obtain
\begin{equation}\label{ineqfinalupperbound}
\begin{split}
\lambda_{2,q}(\Omega)&\leq \frac{\displaystyle\int_\Omega|\nabla f(j_\Omega)|^2\,dx}{\displaystyle\left(\int_\Omega f(j_\Omega)^q\,dx\right)^{\frac{2}{q}}}\\
&\leq \frac{\omega_N^{\frac{2}{q}-1}\,|\Omega|^{1-\frac{2}{q}}}{R_\Omega^2}\,\frac{\left(N\,\omega_N\,\displaystyle \int_0^1|f'(t)|^2\,t^{N-1}\,dt\right)}{\left(N\,\omega_N\,\displaystyle\int_0^1f(t)^q\,t^{N-1}\,dt\right)^{\frac{2}{q}}}\\
&=\frac{\omega_N^{\frac{2}{q}-1}\,|\Omega|^{1-\frac{2}{q}}}{R_\Omega^2}\,\lambda_{2,q}(B_1),
\end{split}
\end{equation}
where in the last equality we used \eqref{sfera}.
By rearranging the terms, it is immediate to see that we have proved the claimed inequality \eqref{HPweakup}.
\par
As for the equality cases, it is easy to see that if equality holds in~\eqref{HPweakup} for an open bounded convex set $\Omega$ containing the origin, then all the inequalities in~\eqref{ineqfinalupperbound} must become equalities and in particular we deduce that
\begin{equation}
\label{fottima}
f\circ j_\Omega \mbox{ is optimal for }\lambda_{2,q}(\Omega),
\end{equation}
and
\begin{equation}
\label{dai}
R_\Omega= \langle x,\nu_\Omega(x)\rangle,\qquad \mbox{ for $\mathcal{H}^{N-1}-$a.\,e. }x\in\partial\Omega.
\end{equation}
By optimality, the first condition \eqref{fottima} implies that $f\circ j_\Omega$ is a weak solution of 
\[
-\Delta u=C\,u^{q-1},\qquad \mbox{ in }\Omega.
\]
In particular, by Elliptic Regularity this implies that $f\circ j_\Omega$ is locally smooth in $\Omega$, say $C^1$. By writing 
\[
j_\Omega(x)=f^{-1}\circ (f\circ j_\Omega)(x),\qquad \mbox{ for }x\in\Omega,
\] 
and observing that $f^{-1}$ is $C^1$ except at zero, we get that 
\[
j_\Omega\in C^1(\Omega\setminus\{0\}).
\]
This in turn implies that a set $\Omega$ attaining the equality in \eqref{HPweakup} is necessarily of class $C^1$. By using the second information \eqref{dai} and the identification of equality cases in Lemma \ref{lm:inradius}, we finally obtain that $\Omega$ must be a ball.
\end{proof}

\begin{oss}
In \cite{BFNT} the following scale invariant quantity has been studied
\[
\frac{\lambda(\Omega)\,T(\Omega)}{|\Omega|},
\]
among the class of open bounded convex sets. In \cite[Theorem 1.4]{BFNT} the authors give the following lower bounds
\[
\frac{\lambda(\Omega)\,T(\Omega)}{|\Omega|}\ge \left(\frac{\pi}{2}\right)^2\,\frac{1}{N^{N+2}\,(N+2)},\qquad \mbox{ for } N\ge 3,
\]
and
\[
\frac{\lambda(\Omega)\,T(\Omega)}{|\Omega|}\ge \left(\frac{\pi}{2}\right)^2\,\frac{1}{12},\qquad \mbox{ for } N=2.
\]
By recalling that in our notation 
\[
\lambda(\Omega)=\lambda_{2,2}(\Omega)\qquad \mbox{ and }\qquad T(\Omega)=\frac{1}{\lambda_{2,1}(\Omega)},
\]
we can rewrite the previous functional as
\[
\frac{\lambda(\Omega)\,T(\Omega)}{|\Omega|}=\frac{\lambda_{2,2}(\Omega)\,R_\Omega^2}{R_\Omega^2\,\lambda_{2,1}(\Omega)\, |\Omega|}.
\]
If we now use the Hersch-Protter inequality to estimate the numerator from below and Theorem \ref{teo:upperbound} with $q=1$ to estimate the denominator from above, we get
\[
\frac{\lambda(\Omega)\,T(\Omega)}{|\Omega|}=\frac{\lambda_{2,2}(\Omega)\,R_\Omega^2}{R_\Omega^2\,\lambda_{2,1}(\Omega)\, |\Omega|}\ge \left(\frac{\pi}{2}\right)^2\,\frac{1}{\lambda_{2,1}(B_1)\,\omega_N}.
\]
By further using that
\[
\frac{1}{\lambda_{2,1}(B_1)\,\omega_N}=\frac{T(B_1)}{\omega_N}=\frac{1}{N\,(N+2)},
\]
we end up with
\[
\frac{\lambda(\Omega)\,T(\Omega)}{|\Omega|}\ge \left(\frac{\pi}{2}\right)^2\,\frac{1}{N\,(N+2)},
\]
which {\it improves strictly} the lower bound of \cite[Theorem 1.4]{BFNT}, in every dimension $N\ge 2$. We refer to \cite[Theorems 1.4 \& 1.5]{BFNT2} for a finer lower bound in a restricted class of convex planar sets, as well as to \cite[Conjecture 4.2]{BBP} for the conjectured sharp lower bound.
\end{oss}

\section{Further estimates}
\label{sect:further}

Our main results can be seen as a double-sided sharp estimate on the shape functional
\[
\Omega\mapsto R^2_{\Omega}\,\lambda_{2,q}(\Omega)\,|\Omega|^\frac{2-q}{q},
\]
in the class of open and bounded convex sets. In a natural way, one could ask whether a similar result can be obtained for the more general shape functional
\[
\Omega\mapsto R^\beta_{\Omega}\,\lambda_{2,q}(\Omega)\,|\Omega|^\alpha,
\]
where $\alpha,\beta\in\mathbb{R}$ are two arbitary exponents. 
\par
Of course, the two exponents $\alpha,\beta$ must satisfy some restrictions. The first one is that of {\it scale invariance}. This imposes that we must have
\[
\beta=2-N\,\left(\alpha-\frac{2-q}{q}\right).
\]
Then we set
\begin{equation}
\label{infolo}
\mathfrak{m}(\alpha):=\inf\left\{R^{2-N\,\left(\alpha-\frac{2-q}{q}\right)}_{\Omega}\lambda_{2,q}(\Omega)\,|\Omega|^\alpha : \Omega\subset\mathbb{R}^N \mbox{ open bounded convex}\right\},
\end{equation}
and 
\begin{equation}
\label{supolo}
\mathfrak{M}(\alpha)=\sup\left\{R^{2-N\,\left(\alpha-\frac{2-q}{q}\right)}_{\Omega}\lambda_{2,q}(\Omega)\,|\Omega|^\alpha : \Omega\subset\mathbb{R}^N \mbox{ open bounded convex}\right\}.
\end{equation}
Observe that our Theorems \ref{teo:lowerbound} and \ref{teo:upperbound} correspond to $\alpha=(2-q)/q$. For the quantity \eqref{infolo} we have the following
\begin{prop}[Minimization]
\label{lm:m}
Let $1\le q<2^*$ and let $\alpha\in\mathbb{R}$ be an exponent. We have 
\[
\mathfrak{m}(\alpha)>0\qquad \Longleftrightarrow\qquad \alpha\ge \max\left\{\frac{2-q}{q},0\right\}.
\]
Moreover, if\,\footnote{Observe that 
\[
\frac{2-q}{q}+\frac{2}{N}>0,
\]
thanks to the fact that $1\le q<2^*$.}
\[
\alpha\ge \frac{2-q}{q}+\frac{2}{N},
\]
balls uniquely minimize \eqref{infolo}, even when the convexity assumption is dropped.
\end{prop}
\begin{proof}
Let us assume that $\mathfrak{m}(\alpha)>0$.
We consider the usual family of sets
\[
\Omega_L=\left(-\frac{L}{2},\frac{L}{2}\right)^{N-1}\times (0,1),
\]
and observe that, taking into account Lemma~\ref{lm:slabtype}, we have the following.\\

\noindent If $1\le q\le 2$ 
\[
R^{2-N\,\left(\alpha-\frac{2-q}{q}\right)}_{\Omega_L}\,\lambda_{2,q}(\Omega_L)\,|\Omega_L|^\alpha\sim C_{N,q,\alpha}\,\frac{\Big(\pi_{2,q}\Big)^2}{L^{(N-1)\,\frac{2-q}{q}}}\,L^{(N-1)\,\alpha}, 
\]
as $L\to+\infty$.\\

\noindent If $2<q<2^*$
\[
R^{2-N\,\left(\alpha-\frac{2-q}{q}\right)}_{\Omega_L}\,\lambda_{2,q}(\Omega_L)\,|\Omega_L|^\alpha\sim C_{N,q,\alpha}\,\lambda_{2,q}(\mathbb{R}^{N-1}\times(0,1))\,L^{(N-1)\,\alpha},
\]
as $L\to+\infty$.

\noindent Thus the assumption $\mathfrak{m}(\alpha)>0$ entails that we must have
\[
\alpha\ge \frac{2-q}{q},\quad \text{if }1\leq q\leq 2,\qquad \mbox{ and }\qquad \alpha\geq 0,\quad \text{if }2<q<2^*,
\]
as desired.
\vskip.2cm\noindent
Let us now assume that
\[
\alpha\ge \max\left\{\frac{2-q}{q},0\right\}.
\]
If $1\le q\le 2$, we can rewrite our functional as follows
\[
R_\Omega^{2-N\,\left(\alpha-\frac{2-q}{q}\right)}\,\lambda_{2,q}(\Omega)\,|\Omega|^\alpha=\left(R^2_\Omega\,\lambda_{2,q}(\Omega)\,|\Omega|^\frac{2-q}{q}\right)\,\left(\frac{|\Omega|}{R_\Omega^N}\right)^{\alpha-\frac{2-q}{q}},
\]
and observe that both terms are bounded from below by a positive constant, in the class of convex sets. For the first one, it is sufficient to use our estimate \eqref{HPweak}, while for the second one we can use that 
\begin{equation}
\label{contiene}
\frac{|\Omega|}{R^N_\Omega}\geq \omega_N.
\end{equation}
If $2<q<2^*$, we can instead rewrite our functional as follows 
\[
R_\Omega^{2-N\,\left(\alpha-\frac{2-q}{q}\right)}\,\lambda_{2,q}(\Omega)\,|\Omega|^\alpha=\left(R^{2-\frac{2-q}{q}\,N}_\Omega\,\lambda_{2,q}(\Omega)\right)\,\left(\frac{|\Omega|}{R_\Omega^N}\right)^\alpha,
\]
and use that for a open convex set $\Omega\subset\mathbb{R}^N$ we have (see \cite[Proposition 6.3]{BraHer})
\[
R^{2-\frac{2-q}{q}\,N}_\Omega\,\lambda_{2,q}(\Omega)\ge C_{N,q},
\]
and again \eqref{contiene}. In both cases, we get $\mathfrak{m}(\alpha)>0$.
\vskip.2cm\noindent
Finally, in the case 
\[
\alpha\ge \frac{2-q}{q}+\frac{2}{N},
\]
it is enough to rewrite the functional in the following form\[
R^{2-N\,\left(\alpha-\frac{2-q}{q}\right)}_{\Omega}\,\lambda_{2,q}(\Omega)\,|\Omega|^\alpha=\left(\lambda_{2,q}(\Omega)\,|\Omega|^{\frac{2}{N}+\frac{2-q}{q}}\right)\,\left(\frac{|\Omega|}{R^N_\Omega}\right)^{\alpha-\left(\frac{2-q}{q}+\frac{2}{N}\right)},
\]
and notice that both quantities are (uniquely) minimized by balls. The first one thanks to the Faber-Krahn inequality and the second one again by \eqref{contiene}. 
The proof is concluded.
\end{proof}
\begin{figure}
\includegraphics[scale=.3]{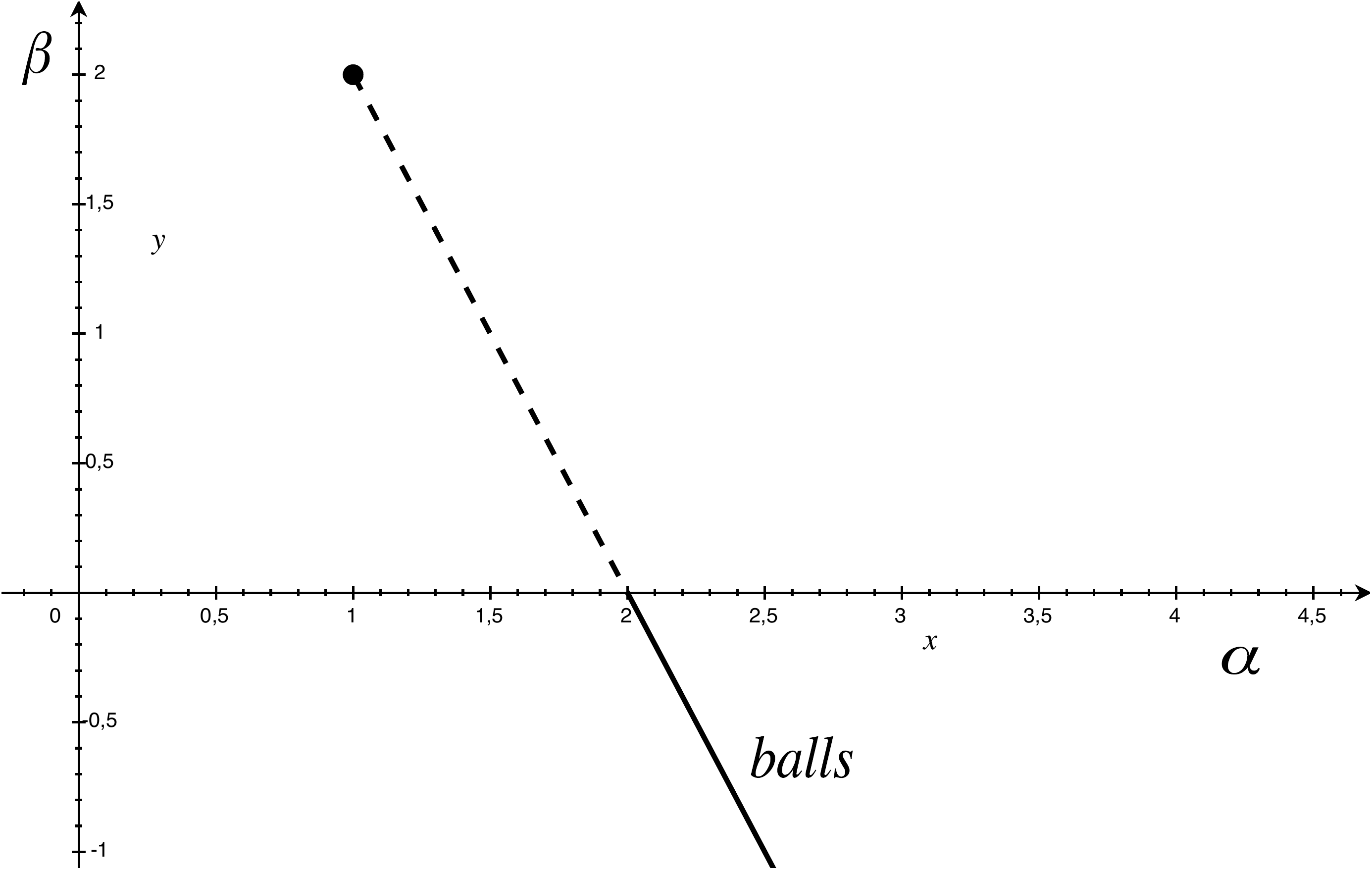}
\caption{A diagram illustrating the result of Lemma \ref{lm:m}, for the minimization of the quantity
\[
\Omega\mapsto R^\beta_{\Omega}\,\lambda_{2,q}(\Omega)\,|\Omega|^\alpha.
\] 
For simplicity we draw it for the case $N=2$ and $q=1$. The black dot corresponds to the case of our Main Theorem, i.e. $\alpha=(2-q)/2=1$ and $\beta=2$. In this case, slab--type sequences give the optimal lower bound \eqref{infolo}. The continuous black line corresponds to the cases where balls are extremals. The dotted line corresponds to the open cases, where the infimum $\mathfrak{m}(\alpha)$ is non-trivial, but its value is not known.}
\end{figure}
For the quantity \eqref{supolo} the situation is simpler and the picture is complete.
\begin{prop}[Maximization]
Let $1\le q<2^*$ and let $\alpha\in\mathbb{R}$ be an exponent. We have
\[
\mathfrak{M}(\alpha)<+\infty\qquad \Longleftrightarrow\qquad \alpha\leq \max\left\{\frac{2-q}{q},0\right\}.
\]
Moreover, in this case balls uniquely attains $\mathfrak{M}(\alpha)$.
\end{prop}
\begin{proof}
Let us assume that $\mathfrak{M}(\alpha)<+\infty$. The asserted restriction on $\alpha$ can then be obtained as before, by considering the family of sets
\[
\Omega_L=\left(-\frac{L}{2},\frac{L}{2}\right)^{N-1}\times (0,1),
\]
we leave the details to the reader.
\vskip.2cm\noindent
We now assume that 
\[
\alpha\leq \max\left\{\frac{2-q}{q},0\right\}.
\]
For $1\le q\le 2$, we rewrite our functional as follows
\[
R^{2-N\,\left(\alpha-\frac{2-q}{q}\right)}_{\Omega}\,\lambda_{2,q}(\Omega)\,|\Omega|^\alpha=\left(R^{2}_{\Omega}\,\lambda_{2,q}(\Omega)\,|\Omega|^\frac{2-q}{q}\right)\,\left(\frac{|\Omega|}{R_\Omega^N}\right)^{\alpha-\frac{2-q}{q}},
\]
and observe that both terms are (uniquely) maximized by balls. The first one thanks to our estimate \eqref{HPweakup}, the second one by \eqref{contiene} (observe that the exponent $\alpha-(2-q)/q$ is non-positive).
\par
For $2<q<2^*$, our assumption on $\alpha$ entails that $\alpha\le 0$ and we use instead the following rewriting
\[
R^{2-N\,\left(\alpha-\frac{2-q}{q}\right)}_{\Omega}\,\lambda_{2,q}(\Omega)\,|\Omega|^\alpha=\left(R^{2-\frac{2-q}{q}\,N}_{\Omega}\,\lambda_{2,q}(\Omega)\right)\,\left(\frac{|\Omega|}{R_\Omega^N}\right)^\alpha.
\]
Here as well, both terms are (uniquely) maximized by balls. For the first one, it is sufficient to use the monotonicity of $\lambda_{2,q}$ with respect to set inclusion, while for the second one we use \eqref{contiene}, as usual.
\end{proof}

\appendix

\section{Some technical results}
The following simple one-dimensional result was an essential ingredient for the proof of the lower bound \eqref{HPweak}.
\begin{lm}
\label{lm:tecnico}
Let $a>0$ and let $\xi:[0,a]\to \mathbb{R}$ be an absolutely continuous function such that
\[
\xi(0)=0\qquad \mbox{ and }\qquad t\mapsto \frac{\xi(t)}{t} \mbox{ is  non-decreasing}.
\]
Let $\psi:[0,a]\to[0,+\infty)$ be a non-increasing function. Then we have
\[
\int_0^a \xi'(t)\,\psi(t)\,dt\le \frac{\xi(a)}{a}\,\int_0^a \psi(t)\,dt.
\] 
\end{lm}
\begin{proof}
Without loss of generality we can suppose that $\psi$ is smooth. By integrating by parts and observing that $\xi(0)=0$, we have
\[
\begin{split}
\int_0^a \xi'(t)\,\psi(t)\,dt&=\xi(a)\,\psi(a)+\int_0^{a} \xi(t)\,(-\psi'(t))\,dt\\
&=\xi(a)\,\psi(a)+\int_0^a \frac{\xi(t)}{t}\,t\,(-\psi'(t))\,dt\\
&\le \xi(a)\,\psi(a)+\int_0^a \frac{\xi(a)}{a}\,t\,(-\psi'(t))\,dt, 
\end{split}
\]
where we also used the monotonicity of both $\xi(t)/t$ and $\psi(t)$. We now further integrate by parts the last integral, so to get
\[
\int_0^a \xi'(t)\,\psi(t)\,dt\le \xi(a)\,\psi(a)-\frac{\xi(a)}{a}\,a\,\psi(a)+\frac{\xi(a)}{a}\,\int_0^a \psi(t)\,dt.
\]
This concludes the proof.
\end{proof}

\begin{lm}
\label{lm:slabtype}
Let $N\ge 2$ and let $1\le q<2^*$, for every $L>0$ we set
\[
\Omega_L=\left(-\frac{L}{2},\frac{L}{2}\right)^{N-1}\times (0,1).
\]
Then we have:
\begin{enumerate}
\item for $1\le q\le 2$
\[
\lim_{L\to +\infty} L^{(N-1)\,\frac{2-q}{q}}\,\lambda_{2,q}(\Omega_L)=\Big(\pi_{2,q}\Big)^2,\qquad \mbox{ as } L\to +\infty.
\]
\item for $2<q<2^*$
\[
\lim_{L\to+\infty} \lambda_{2,q}(\Omega_L)=\lambda_{2,q}(\mathbb{R}^{N-1}\times(0,1))>0. 
\]
\end{enumerate}
\end{lm}
\begin{proof}
We distinguish again the two cases.
\vskip.2cm\noindent
{\bf Case $1\le q\le 2$.} For $q=2$ this is contained for example in \cite[Lemma A.2]{BraBus}, we thus focus on the case $q<2$. By \cite[equation (3.6)]{BraBus}, we have
\[
\lim_{L\to+\infty}\lambda_{2,q}(\Omega_L)\,\left(\frac{|\Omega_L|^{\frac{1}{2}+\frac{1}{q}}}{P(\Omega_L)}\right)^2=\left(\frac{\pi_{2,q}}{2}\right)^2,
\]
where $P(\Omega_L)$ stands for the perimeter of $\Omega_L$.
If we now use that
\[
|\Omega_L|= L^{N-1}\qquad \mbox{ and }\qquad  P(\Omega_L)\sim 2\,L^{N-1},\quad \mbox{ as } L\to +\infty,
\]
we get the desired result.
\vskip.2cm\noindent
{\bf Case $2<q<2^*$.} By monotonicity of $\lambda_{2,q}$ with respect to set inclusion, we have that 
\[
\lambda_{2,q}(\Omega_L)\ge \lambda_{2,q}(\mathbb{R}^{N-1}\times (0,1))\quad \mbox{and}\quad L\mapsto \lambda_{2,q}(\Omega_L) \mbox{ is monotone decreasing}.
\]
Thus we get that 
\[
\lim_{L\to+\infty}\lambda_{2,q}(\Omega_L)\ge \lambda_{2,q}(\mathbb{R}^{N-1}\times (0,1)).
\]
On the other hand, for every $\varepsilon>0$ we can take $\varphi_\varepsilon\in C^\infty_0(\mathbb{R}^{N-1}\times(0,1))$ such that
\[
\lambda_{2,q}(\mathbb{R}^{N-1}\times (0,1))+\varepsilon\ge \frac{\displaystyle\int_{\mathbb{R}^{N-1}\times(0,1)} |\nabla \varphi_\varepsilon|^2\,dx}{\displaystyle\left(\int_{\mathbb{R}^{N-1}\times(0,1)} |\varphi_\varepsilon|^q\,dx\right)^\frac{2}{q}}.
\]
Since $\varphi_\varepsilon$ has compact support, for $L$ large enough we get that $\varphi_\varepsilon\in C^\infty_0(\Omega_L)$, as well. This shows that for every $\varepsilon>0$
\[
\lambda_{2,q}(\mathbb{R}^{N-1}\times (0,1))+\varepsilon\ge \lim_{L\to+\infty}\lambda_{2,q}(\Omega_L),
\]
and thus the claimed convergence of $\lambda_{2,q}(\Omega_L)$ follows. 
\par
We are only left with showing that $\mathbb{R}^{N-1}\times (0,1)$ has a non-trivial Poincar\'e-Sobolev constant $\lambda_{2,q}$. By recalling that for a open convex set $\Omega\subset\mathbb{R}^N$ we have (see \cite[Proposition 6.3]{BraHer})
\[
\lambda_{2,q}(\Omega)\ge \frac{C_{N,q}}{R_\Omega^{2+\frac{2-q}{q}\,N}},
\]
we get the desired assertion.
\end{proof}

\end{document}